\documentclass[a4paper,numbook,babel,final,envcountsame,11pt]{article}

\usepackage{amssymb}
\usepackage{amsthm}
\usepackage{amsmath}
\usepackage{graphicx}
\usepackage{array,hhline}
\numberwithin{figure}{section}

\newcommand{\Aset}{\mathcal{A}}
\newcommand{\Bset}{\mathcal{B}}
\newcommand{\Cset}{\mathcal{C}}

\newcommand{\Xset}{\mathcal{X}}
\newcommand{\Yset}{\mathcal{Y}}
\newcommand{\Zset}{\mathcal{Z}}

\newtheorem{thm}{Theorem}[section]
\newtheorem{cor}{Corollary}[section]
\newtheorem*{cor*}{Corollary}
\newtheorem{defn}[thm]{Definition}
\newtheorem{rem}[thm]{\bf{Remark}}

\newtheorem{proposition}[thm]{Proposition}
\newtheorem{lem}[thm]{Lemma}
\begin{document}

     % Enter your date or \today between curly braces

\title{Sets of lines passing through one point each of three sets
of points}

\title{\Large{\bf On \ plane \ algebraic curves \ passing \ through\\ $n$-independent \ nodes}}
\author{H. A. Hakopian, \ H. M. Kloyan, \ D. S. Voskanyan \\ \\
\emph{Department of Informatics and Applied Mathematics,}\\  \emph{Yerevan State University}}

\date{}

\maketitle

%%%%%%%%%%%%%%%%%%%%%%%%%%%%%%%%%%%%%%%%%%%%%%%%%%%%%%%%%%%%%%%%%%%%%%%%
\begin{abstract}
 Let a set of nodes $\mathcal X$ in the plane be $n$-independent, i.e., each node has a fundamental polynomial of degree $n.$ Assume that\\
$\#\mathcal X=d(n,k-3)+3= (n+1)+n+\cdots+(n-k+5)+3$ and $4 \le k\le n-1.$ In this paper we prove that there are at most seven linearly independent curves of degree less than or equal to $k$ that pass through all the nodes of $\mathcal X.$ We provide a characterization of the case when there are exactly seven such curves. Namely, we prove that then the set $\mathcal X$ has a very special construction: all its nodes but three belong to a (maximal) curve of degree $k-3.$ Let us mention that in a series of such results this is the third one. At the end an important application to the bivariate polynomial interpolation is provided, which is essential also for the study of the Gasca-Maeztu conjecture.
interpolation, fundamental polynomial, $n$-independent nodes.
\end{abstract}

\parindent=1cm

\section{Introduction}

 Denote the space of all bivariate polynomials of total degree not exceeding $n$ by $$\Pi_n= \left\{\sum_{i+j\leq n} a_{ij} x^i y^j \right\}.$$
We have that
\begin{equation*}
  N := N_n := \dim \Pi_n = (1/2)(n+1)(n+2).
\end{equation*}
Denote by $\Pi$ the space of all bivariate polynomials.

Consider a set of $s$ distinct nodes ${\mathcal X}= {\mathcal X}_s = \{(x_1,y_1), (x_2,y_2), \ldots, (x_s,y_s) \}.$
 The problem of finding a polynomial $p \in \Pi_n,$ which satisfies the conditions
\begin{equation}\label{eq:intpr}
  p(x_i,y_i) = c_i, \quad i=1, \ldots, s,
\end{equation}
is called interpolation problem.

A polynomial $p \in \Pi_n$ is called a fundamental polynomial for a node \linebreak$A\in{\mathcal X}$ if
$p(A)=1\ \ \text{and}\ \ p\big\vert_{{\mathcal X}\setminus\{A\}} = 0,$ where $p\big\vert_{\mathcal X}$ means the restriction of $p$ on $\mathcal X.$ We denote this $n$-fundamental polynomial
by $p_{A}^\star:=p_{A,\Xset}^\star.$ 

\begin{defn} \label{poised}
\textit {The interpolation problem with a set of nodes ${\mathcal X}_s$ is called $n$-poised if for any data $(c_1,\ldots, c_s)$
there is a unique polynomial $p\in\Pi_n$ satisfying the interpolation conditions \eqref{eq:intpr}.}
\end{defn}
A necessary condition of
poisedness is $\#{\mathcal X}_s=s = N.$

Next, let us consider the concept of $n$-independence (see \cite{E,HJZ}).

\begin{defn}
\textit {A set of nodes ${\mathcal X}_s$ is called $n$-independent, if all its
nodes have $n$-fundamental polynomials. Otherwise, it is called $n$-dependent.}
\end{defn}

Fundamental polynomials are linearly independent.~Therefore a necessary condition of $n$-independence for ${\mathcal X}_s$
is $s \leq N$.

\subsection{Some properties of $n$-independent nodes}

 Let us start with the following
\begin{lem}[Lemma 2.2, \cite{HM}] \label{XA}
\textit {Suppose that a set of nodes $\Xset$ is $n$-independent and the nodes of another set $\Yset$ have
$n$-fundamental polynomials with respect to the set $\Zset = \Xset\cup\Yset.$ Then the set $\Zset$ is $n$-independent too.}
\end{lem}

Denote the distance between the points $A$ and $B$ by $\rho(A,B).$ Let us recall the following (see \cite{H00})
\begin{lem} \label{eps'}
\textit {Suppose that ${\mathcal X}_s=\{A_i\}_{i=1}^s$ is an $n$-independent set. Then there is a number $\epsilon>0$ such that
any set ${\mathcal X}_s'=\{A_i'\}_{i=1}^s,$ with the property that  $\rho(A_i,A_i')<\epsilon,\ i=1,\ldots,s,$  is $n$-independent too.}
\end{lem}

Next result concerns the extensions of $n$-independent sets.
\begin{lem}[Lemma 2.1, \cite{HJZ}]\label{ext}
\textit {Any $n$-independent set ${\mathcal X}$ with $\#{\mathcal X}<N$ can be enlarged to an $n$-poised set.}
\end{lem}

Denote the linear space of polynomials of total degree at most $n$
vanishing on ${\mathcal X}$ by
\begin{equation*}{{\mathcal P}}_{n,{\mathcal X}}:=\left\{p\in \Pi_n
: p\big\vert_{\mathcal X}=0\right\}.
\end{equation*}
The following two propositions are well-known  (see, e.g.,\cite{HJZ}).
\begin{proposition} \label{PnX}
\textit {For any node set ${\mathcal X}$ we have that
\begin{equation*}\label{eq:theta1} \dim {{\mathcal P}}_{n,{\mathcal X}} = N - \#{\mathcal Y},\end{equation*}
where ${\mathcal Y}$ is a maximal $n$-independent subset of ${\mathcal X}.$}
\end{proposition}
\begin{proposition}\label{maxline}
\textit {If a polynomial $p\in \Pi_n$ vanishes at $n+1$ points of  a line $\ell$, then we have that $p\big\vert_{\ell}=0$ and 
$p = \ell r,$ where $r \in \Pi_{n-1}.$}
\end{proposition}
A plane algebraic curve is the zero set of some bivariate polynomial of degree $\ge 1.$~To simplify notation, we shall use the same letter,  say $p$,
to denote the polynomial $p$ and the curve given by the equation $p(x,y)=0$.

In the sequel we will need the following
\begin{proposition}[Prop. 1.10, \cite{HM}]\label{hm}
Let $\Xset$ be a set of nodes. Then ${{\mathcal P}}_{n,{\mathcal X}}=\{0\}$ if and only if  
$\Xset$ has an $n$-poised subset.
\end{proposition}

Set $d(n, k) := N_n - N_{n-k} = (1/2) k(2n+3-k).$
The following is a  generalization of Proposition \ref{maxline}.
\begin{proposition}[Prop. 3.1, \cite{Raf}]\label{maxcurve}
\textit {Let $q$ be an algebraic curve of degree $k \le n$ with no multiple components. Then the following hold:}

\textit {$i)$ any subset of $q$ containing more than $d(n,k)$ nodes is
$n$-dependent;}

\textit {$ii)$ any subset ${\mathcal X}$ of $q$ containing exactly $d(n,k)$ nodes is $n$-independent if and only if the following condition
holds:
\begin{equation}\label{p=qr}
p\in {\Pi_{n}}\quad \text{and} \quad p|_{{\mathcal X}} = 0 \implies  p = qr,\ \hbox{where}\ r \in \Pi_{n-k}.
\end{equation}}
\end{proposition}
Thus, according to Proposition~\ref{maxcurve}, $i)$, at most $d(n,k)$ $n$-independent nodes can lie in a curve $q$ of degree $k \le n$.
This motivates the following
\begin{defn}[Def. 3.1, \cite{Raf}]\label{def:maximal}
\textit {Given an $n$-independent set of nodes $\mathcal X$ with $\#\mathcal X\ge d(n,k).$ A curve of degree $k \le n$ passing through $d(n,k)$ points
of $\mathcal X$ is called maximal.}
\end{defn}

Let us bring a characterization of maximal curves:

\begin{proposition}[Prop. 3.3, \cite{Raf}] \label{maxcor}
\textit {Given an $n$-independent set of nodes $\mathcal X$ with $\#\mathcal X\ge d(n,k).$ Then
a curve $\mu$ of degree $k,\ k\le n,$ is a maximal curve if and only if }
$p\in\Pi_n,\ p|_{\mathcal X \cap\mu}=0 \implies p=\mu s,\ s\in\Pi_{n-k}.$
\end{proposition}

Next result concerns  maximal independent sets in curves.
\begin{proposition}[Prop. 3.5, \cite{HakTor}] \label{extcurve}
\textit {Assume that $\sigma$ is an algebraic curve  of degree $k$ with no multiple components and ${\mathcal X}_s\subset \sigma$ is any $n$-independent node set of cardinality $s,\ s<d(n,k).$ Then the set ${\mathcal X}_s$ can be enlarged to a maximal $n$-independent set ${\mathcal X}_{d}\subset \sigma$ of cardinality $d=d(n,k)$.}
\end{proposition}

Below a replacement of a node in an $n$-independent set is described such that the set remains $n$-independent.
\begin{lem}[Lemma 6, \cite{HK}] \label{lem44}
\textit {Assume that ${\mathcal X}$ is an $n$-independent node set and a node $A\in \mathcal X$ has an $n$-fundamental polynomial $p^\star_A$ such that $p^\star_A(A')\neq 0.$ Then we can replace the node $A$ with $A'$ such that the resulted set ${\mathcal X}':={\mathcal X}\cup\{A'\}\setminus \{A\}$
is $n$-independent too. In particular, such replacement can be done in the following two cases:}

$i)$ if a node $A\in \mathcal X$ belongs to several components of $\sigma,$ then we can replace it with a node $A',$ which belongs to only one (desired) component,

$ii)$ if a curve $q$ is not a component of an $n$-fundamental polynomial $p^\star_A$ then we can replace the node $A$ with a node $A'$ lying in $q.$
\end{lem}

Next result from Algebraic Geometry will be used in the sequel:
\begin{thm}[Th.~2.2, \cite{W}] \label{intn} If $\Cset$ is a curve of degree $n$ with no multiple components, then through any point $O$ not in $\Cset$ there pass lines which intersect $\Cset$ in $n$ distinct points.
\end{thm}
\noindent Let us mention also that, as it follows from the proof, if a line $\ell$ through a point $O$ intersects $\Cset$ in $n$ distinct points then any line through $O,$ sufficiently close to $\ell,$ has the same property.
Finally, let us present a well-known
\begin{lem} \label{2cor}
\textit {Suppose that $m$ linearly independent polynomials vanish at the set ${\mathcal X}.$
Then for any node $A\notin {\mathcal X}$ there are $m-1$ linearly independent polynomials, in their linear span, vanishing at $A$ and the set ${\mathcal X}.$}
\end{lem}

\section {Main results and a series of results}
Let us start with the first result of a series of results:
\begin{thm}[Th.~1, \cite{HT}] \label{ht}
\textit {Assume that ${\mathcal X}$ is an $n$-independent set of $d(n, k-1)+2$ nodes lying in a curve of degree $k$ with $k\le n.$ Then the curve
is determined uniquely by these nodes.}
\end{thm}
The second result in this series is the following
\begin{thm}[Th.~4.2,  \cite{HakTor}]\label{ht2}
\textit {Assume that ${\mathcal X}$ is an $n$-independent set of $d(n, k-1)+1$ nodes with $2\le k\le n-1.$ Then at most two different curves of degree
$\le k$ may pass through all the nodes of ${\mathcal X}.$ Moreover, there are such two curves for the set ${\mathcal X}$ if and only if all the nodes of ${\mathcal X}$ but one lie in a maximal curve of degree $k-1.$}
\end{thm}
Next result is the following
\begin{thm}[Th.~3,  \cite{HK}]\label{hk}
\textit {Assume that ${\mathcal X}$ is an $n$-independent set of $d(n, k-2)+2$ nodes with $3\le k\le n-1.$ Then at most four linearly independent curves of degree
$\le k$ may pass through all the nodes of ${\mathcal X}.$ Moreover, there are such four curves for the set ${\mathcal X}$ if and only if all the nodes of ${\mathcal X}$ but one lie in a maximal curve of degree $k-2.$}
\end{thm}

Now let us present the main result of this paper:
\begin{thm}\label{mainth}
\textit {Assume that ${\mathcal X}$ is an $n$-independent set of $d(n, k-3)+3$ nodes with $4\le k\le n-1.$ Then at most seven linearly independent curves of degree
$\le k$ may pass through all the nodes of ${\mathcal X}.$ Moreover, there are such seven curves for the set ${\mathcal X}$ if and only if all the nodes of ${\mathcal X}$ but three lie in a maximal curve of degree $k-3.$}
\end{thm}

Let us mention that the inverse implication in the ``Moreover'' part is straightforward. Indeed, assume that $d(n, k-3)$ nodes of ${\mathcal X}$ are located in a curve $\mu$ of degree $k-3$.
Therefore, the curve $\mu$ is maximal and the remaining three nodes of ${\mathcal X},$ denoted by $A, B$ and $C,$ are outside of it: $A,B,C\notin \mu.$ Hence, in view of Proposition \ref{maxcor}, we have that

\noindent $\quad{\mathcal P}_{k,{\mathcal X}}=\left\{{p : p\in \Pi_{k}}, p_{\mathcal X}=0\right\}= \left\{{q\mu : q\in \Pi_{3}}, q(A)=q(B)=q(C)=0\right\}.$

\noindent Thus we get readily that
$\dim{\mathcal P}_{k,{\mathcal X}}=\dim\left\{{q\in \Pi_{3}}: q(A)=q(B)=q(C)=0\right\}$ $=\dim{\mathcal P}_{3, \{A,B,C\}}=10-3=7. $
Note that in the last equality we use Proposition \ref{PnX} and the fact that any three nodes are $3$-independent.

We get also that it is enough to prove only the ``Moreover'' part. Indeed, assume that the ``Moreover'' part is proved.  Assume also that there are $\ge 7$ linearly independent curves satisfying the conditions mentioned in Theorem \ref{mainth}.  Then, as we showed above, we have that $\dim{\mathcal P}_{k,{\mathcal X}}=7,$ i.e., there are exactly $7$ such curves, Q.E.D.

It is worth mentioning that in the proof of Theorem \ref{mainth} we prove some interesting version of Theorem \ref{hk}, where we increase  the number of nodes by one and decrease  the number of linearly independent curves by one:

\begin{thm}\label{hkv}
\textit {Assume that ${\mathcal X}$ is an $n$-independent set of $d(n, k-2)+3$ nodes with $3\le k\le n-2.$ Then at most three linearly independent curves of degree
$\le k$ may pass through all the nodes of ${\mathcal X}.$ Moreover, there are such three curves for the set ${\mathcal X}$ if and only if all the nodes of ${\mathcal X}$  lie in a  curve of degree $k-1,$ or all the nodes of  ${\mathcal X}$ but three lie in a (maximal) curve of degree $k-2.$}
\end{thm}

 \section{Some preliminaries}

We will start the proof of Theorem \ref{mainth} in Section \ref{proofs}. Since then we need to do considerable amount of preliminary work.
\begin{lem}\label{mainl}
\textit {Assume that the hypotheses of Theorem \ref{mainth} hold and assume additionally that there is a curve $\sigma_0\in \Pi_{k-2}$ passing through all the nodes of ${\mathcal X}.$
Then all the nodes of ${\mathcal X}$ but three (collinear) lie in a maximal curve $\mu$ of degree $k-3.$}
\end{lem}
\begin{proof}
First note that the curve $\sigma_0$ is of exact degree $k-2,$ since it passes through more than $d(n, k-3)$ $n$-independent nodes. This implies also that $\sigma_0$ has no multiple components.
Therefore, in view of Proposition \ref{extcurve}, we can enlarge the set ${\mathcal X}$ to  a maximal $n$-independent set ${\mathcal Z}\subset \sigma_0,$ by adding $d(n,k-2)-d(n,k-3)-3=n-k+1$ nodes, i.e.,
${\mathcal Z} = {\mathcal X}\cup{\mathcal A},\ \hbox{where}\ {\mathcal A}=\{A_0,\ldots,A_{n-k}\}.$

In view of Lemma \ref{lem44}, $i)$, we may suppose that the nodes from ${\mathcal A}$ are not intersection points of the components of $\sigma_0.$

Next, we are going to prove that these $n-k+1$ nodes are collinear together with $m\ge 3$ nodes from ${\mathcal X}.$
To this end denote the line through the nodes $A_0$ and $A_{1}$ by $\ell_{01}.$ Then for each $i=2\ldots,n-k,$ choose a line $\ell_i$ passing through the node $A_{i},$ which is not a component of $\sigma_0.$ We require also that $\ell_i$ does not pass through other nodes of $\Aset$ and therefore the lines are distinct.

 Now suppose that $\sigma^*\in\Pi_k$ vanishes on ${\mathcal X}.$  Consider the polynomial  $p=\sigma^*\ell_{01}\ell_2\cdots\ell_{n-k}.$ We have that $p\in\Pi_n$ and $p$ vanishes on the node set ${\mathcal Z},$ which is a maximal $n$-independent set in the curve $\sigma_0.$ Therefore, we obtain that 
$$p=\sigma^*\ell_{01}\ell_2\cdots\ell_{n-k}=\sigma_0 r,\ \hbox{where}\ r\in\Pi_{n-k+2}.$$
The lines $\ell_i,\ i=2,\ldots,n-k,$ are not components of $\sigma_0.$ Therefore, they are  components of the polynomial $r.$ Hence we obtain that
$$\sigma^*\ell_{01}=\sigma_0\gamma,\ \hbox{where}\ \gamma\in\Pi_{3}.$$
Now let us verify that $\ell_{01}$ is a component of $\sigma_0.$ Indeed, otherwise it is a component of the cubic $\gamma$ and we get that
$$\sigma^*\in \Pi_k, \ \sigma^*\big\vert_{{\mathcal X}} = 0
 \ \implies\ \sigma^*=\sigma\beta,\ \hbox{where}\ \beta\in\Pi_{2}.$$
Therefore, we obtain that $\dim{\mathcal P}_{k,\mathcal X}\le 6,$ which contradicts the hypothesis.

Thus we have that
\begin{equation}\label{01k-3}\sigma_0=\ell_{01}\sigma_{k-3},\ \hbox{where}\ \sigma_{k-3}\in\Pi_{k-3}.\end{equation}

Now let us show that all the nodes of ${\mathcal A}$ belong to $\ell_{01}.$
Suppose conversely that a node from ${\mathcal A},$ say $A_2,$ does not belong to the line $\ell_{01}.$ Then in the same way
as in the case of the line $\ell_{01}$ we get that $\ell_{02}$ is a component of $\sigma_0.$ Therefore the node $A_0$ is an intersection point of two components of $\sigma_0,$ i.e., $\ell_{01}$
and $\ell_{02},$ which contradicts our assumption.

Thus we get that $\Aset\subset \ell_{01}.$ Note that $\ell_{01}$ is not a component of $\sigma_{k-3}$ since then it will be a multiple component of $\sigma_0.$ 

Next, let us verify that when enlarging the set $\Xset\subset \sigma_0$ to an $n$-maximal set one has to locate the added nodes outside the component $\sigma_{k-3}.$ Indeed, what was proved already implies that the only possible location of such a node in $\sigma_{k-3}$ is an intersection point with $\ell_{01}.$
 But in the latter case, by using Lemma \ref{lem44},  we can replace the node, say $A_1,$ with one belonging only to the component $\sigma_{k-3},$ say $A_1',$which is a contradiction. Indeed, again $A_0$ is the intersection point of two components of $\sigma_0,$ the line through $A_0, A_1$ and the line through $A_0, A_1'.$

Hence, in view of Proposition \ref{extcurve} we get that $\mu=\sigma_{k-3}$ is a maximal curve for $\Xset.$ Therefore, it vanishes at exactly $d(n,k-3)$ nodes of $\Xset.$ The remaining three nodes, according to \eqref{01k-3}, belong to the line $\ell_{01}.$
\end{proof}

The next result we prove with tools of mathematical analysis. 
\begin{proposition}\label{eps} Assume that $p_1,p_2\in\Pi,$  $\deg p_2\le \deg p_1+1,$ and $p_1$ has no multiple factors. Then, for sufficiently small $\epsilon,$ the polynomial $p_1+\epsilon p_2$ has no multiple factors either.
\end{proposition}
\begin{proof} Assume by way of contradiction that there is a  sequence $\epsilon_n$ such that
\begin{equation}\label{qr}p_1+\epsilon_n p_2=q_n r_n^2,\ \hbox{where}\ q_n,r_n\in \Pi,\ \deg r_n\ge 1,\hbox{and}\ \epsilon_n\rightarrow 0.\end{equation}
We have that $\deg(p_1+\epsilon_n p_2)\le\max (\deg p_1,\deg p_2),$ and hence
\begin{equation}\label{degp+1}\deg q_n+2 \deg r_n\le\max (\deg p_1,\deg p_2)\le \deg p_1+1.\end{equation}
We deduce from here that there is a subsequence $n_k$ such that
$$\deg q_{n_k}=m_1=const. \ \hbox{and}\ \deg r_{n_k}=m_2=const.$$

Without loss of generality assume that 
\begin{equation}\label{sub}\{\epsilon_n\}\equiv \{\epsilon_{n_k}\}.\end{equation}
Thus we have that
$$q_{n}=\sum_{i+j\le m_1}a_{ij}^{(n)}x^iy^j,\quad r_{n}=\sum_{i+j\le m_2}b_{ij}^{(n)}x^iy^j.$$
In view of \eqref{qr}, by a normalization of $r_{n},$ i.e., by multiplying it by a constant $c$ and dividing $q_n$ by $c^2$, we may assume that
\begin{equation}\label{b1}\max|b_{ij}^{(n)}|=1\  \forall n.\end{equation}

Now, let us denote $M_n:=\max|a_{ij}^{(n)}|.$

{Case 1.} Assume that (a subsequence of) $M_n$ is bounded: $M_n\le M.$
Note that in the case of the subsequence we may use again a replacement \eqref{sub} and have that the whole sequence $M_n$ is bounded.
In this case, by using the Bolzano–Weierstrass theorem, we have for a subsequence  $\{n_k\}$  that
$$a_{ij}^{(n_k)}\rightarrow a_{ij}^0\ \hbox{and}\ b_{ij}^{(n_k)}\rightarrow b_{ij}^0, \ \forall i,j.$$
Here, we use the fact that the number of the coefficients is finite.

By setting $n=n_k$ in \eqref{qr} and tending $k\rightarrow \infty$ we obtain that
$p_1=q_0 r_0^2,$
where
$$q_{0}=\sum_{i+j\le m_1}a_{ij}^{0}x^iy^j,\quad r_{0}=\sum_{i+j\le m_2}b_{ij}^{0}x^iy^j.$$
This contradicts the hypothesis for $p_1$ if $\deg r_0\ge 1.$

Let us verify the latter inequality. Since $\deg r_n\ge 1,$ we get from \eqref{degp+1} that $\deg q_n\le \deg p_1-1.$
Therefore $m_1\le \deg p_1 -1$ and hence $\deg r_0\ge 1.$

{Case 2.} By taking into account a replacement \eqref{sub} it remains to consider the case $M_n\rightarrow +\infty.$

There are numbers $i_0,j_0,i_1,j_1$ and a subsequence $n=\{n_k\},$ such that
\begin{equation}\label{ab0}|a_{i_0j_0}^{(n_k)}|=\max_{i,j} |a_{ij}^{(n_k)}|\ \hbox{and}\ |b_{i_1j_1}^{(n_k)}|= \max_{i,j} |b_{ij}^{(n_k)}|=1\ \forall k.\end{equation}
Here, again we use   the fact that the number of the coefficients is finite. In the last equality we use \eqref{b1}.

Now, let us set $n=n_k$ in \eqref{qr} and divide both sides by $M_{n_k}$ to get
 \begin{equation}\label{qr'}{1\over {M_{n_k}}} p_1+{\epsilon_{n_k}\over {M_{n_k}}} p_2= \left({1\over{M_{n_k}}}q_{n_k}\right) r_{n_k}^2.\end{equation}
 Evidently, the left hand side here tends to zero. For the right hand side we have that the coefficients of the polynomials ${1\over{M_{n_k}}}q_{n_k}$ and $r_{n_k}$ are bounded by $1.$ As above by using  the Bolzano–Weierstrass theorem and passing to a new subsequence $\{n_k'\}\subset
\{n_k\}$  we obtain that
$${1\over{M_{n_k'}}}a_{ij}^{({n_k'})}\rightarrow a_{ij}^*\ \hbox{and}\ b_{ij}^{({n_k'})}\rightarrow b_{ij}^*, \ \forall i,j.$$ In view of \eqref{ab0} we have that
\begin{equation}\label{apq}|a_{i_0j_0}^*|=1\ \hbox{and}\ |b_{i_1j_1}^*|=1.
\end{equation}
Now, by setting $n=n_k'$ in \eqref{qr} and tending $k\rightarrow \infty$ we get that
$0=q_* r_*^2,$
where
$$q_{*}=\sum_{i+j\le m_1}a_{ij}^{*}x^iy^j,\quad r_{*}=\sum_{i+j\le m_2}b_{ij}^{*}x^iy^j.$$
In view of \eqref{apq} this is a contradiction.
\end{proof}
\begin{rem} In the same way one can prove the following statement:
\label{rem} Assume that $p_1,p_2\in\Pi,\ \deg p_2\le \deg p_1,$ and $p_1$ is not reducible. Then, for sufficiently small $\epsilon,$ the polynomial $p_1+\epsilon p_2$  is not reducible either.
\end{rem}
\noindent Note that, as the example of $p_2=xp_1$ shows, the condition $\deg p_2\le \deg p_1$  is essential here.

\noindent Next result will help to make  the hypotheses of Theorem \ref{mainth} more precise.
\begin{proposition} \label{nomult} Suppose that there are seven linearly independent polynomials from $\Pi_k$ vanishing on a set ${\mathcal X}.$
Then, there are seven linearly independent polynomials  vanishing on a set ${\mathcal X},$ each of which is of exact degree $k$ and has no multiple factors, or, alternatively there are three  linearly independent polynomials from $\Pi_{k-1}$ vanishing on ${\mathcal X}.$
\end{proposition}
\begin{proof} Let $\sigma_{i}\in\Pi_k, 0\le i\le 6,$ be the given polynomials. We may assume that a polynomial, say $\sigma_0,$ is of exact degree $k.$ Indeed, if the degree of each of  seven polynomials is less than $k$ then the conclusion of Proposition holds.

Therefore  we may assume that all the polynomials $\sigma_{i}, 0\le i\le 6$ are of exact degree $k.$  Indeed,  it suffices to replace these polynomials with the seven polynomials $\sigma_0$ and $\sigma_0+\epsilon\sigma_i, 1\le i\le 6,\ \forall \epsilon\neq 0.$

Next, let us prove that a polynomial, say $\sigma_{0},$ has no multiple factors. Indeed
assume conversely that each of the seven polynomials has a multiple factor. In view of Lemma \ref{mainl} the multiple factors are lines with multiplicity two. Thus, we have that
\begin{equation}\label{l2}\sigma_i=\ell_i^2q_i,\ 0\le i \le 6,\ \hbox{where}\ \ell_i\in\Pi_1,\  q_i\in\Pi_{k-2}.\end{equation}
Then we replace
these polynomials with the seven polynomials $\check\sigma_i=\ell_i q_i\in\Pi_{k-1},\ 0\le i \le 6,$ which clearly vanish at the node set $\Xset.$ Let us verify that among these latter seven polynomials there are at least three linearly independent ones. Conversely assume that the seven polynomials are linear combinations of two  of them, say $\check\sigma_i=\ell_iq_i,\ i=0,1.$ Then we get readily that the seven linearly independent polynomails in \eqref{l2} are linear combinations of the following six polynomials:
$$\check\sigma_i,x \check\sigma_i, y\check\sigma_i,\  i=0,1,$$
which is a contradiction.
Indeed, assume that $\ell_i=A_ix+B_iy+C_i,\ i=0,\ldots,6.$ Then  for $i=0,1,$ we have that
$$
\sigma_i=\ell_i^2q_i=(A_ix+B_iy+C_i)\check\sigma_i=Ax\check\sigma_i+B_iy\check\sigma_i+C_i\check\sigma_i.$$
\noindent Now, assume that $\check\sigma_i=a_i\check\sigma_0+b_i\check\sigma_1,$  for $i=2,\ldots,6.$ Then we have that
$$\check\sigma_i=\ell_i^2q_i=(A_ix+B_iy+C_i)\check\sigma_i=
 (A_ix+B_iy+C_i)(a_i\check\sigma_0+b_i\check\sigma_1)$$
 $$=a_iA_ix\check\sigma_0+a_iB_iy\check\sigma_0
+a_iC_i\check\sigma_0+b_iA_ix\check\sigma_1+b_iB_iy\check\sigma_1+b_iC_i\check\sigma_1.$$

Finally, by assuming that  $\sigma_{0},$ has no multiple factors, let us again replace the seven polynomials $\sigma_{i}, 0\le i\le 6,$ with the seven polynomials $\sigma_0$ and $\sigma_0+\epsilon\sigma_i, 1\le i\le 6,$ for a sufficiently small $\epsilon>0.$
This, in view of Proposition \ref{eps}, completes the proof.
\end{proof}

\begin{proposition} \label{factor3} Suppose that $\sigma_i,\ i=0,\ldots,6,$ are linearly independent polynomials of exact degree $k$ and have no multiple factors. Then there is a polynomial in the linear span of  $\sigma_i,\ i=1,\ldots,6,$ which has no multiple factors and differs from $\sigma_0$ with a factor of degree at least three.
\end{proposition}
Let us prove first
\begin{lem} \label{nom} Let $\sigma_0, s_1,s_2,$ be linearly independent polynomials of exact degree $k,$ with no multiple factors. Suppose also that any linear combination of  $s_i,\ i=1,2,$ differs from $\sigma_0$ with a factor from $\Pi_2.$ Then we have that
\begin{equation} \label{sb} \sigma_0=\tilde \sigma_0\beta_0, \ \  s_1=\tilde \sigma_0\beta_1,\ \  s_2=\tilde \sigma_0\beta_2,\ \hbox{where}\ \tilde \sigma_0\in\Pi_{k-1},\ \beta_i\in \Pi_2.\end{equation}
Moreover, $\tilde\sigma_0$ is uniquely determined
 from the first two relations  here, if $\beta_0$ and $\beta_1$ are relatively prime.
 
\noindent Furthermore, if $\beta_0$ has a common factor with $\beta_1$ and a common factor with $\beta_2$ then the following alternative takes place:
Either,

(i) $\beta_i=\ell\ell_i,\ i=0,1,2,$ i.e., they have a common linear factor, or

(ii) $\beta_0$ and $\beta_1+\epsilon\beta_2$ are relatively prime $\forall \epsilon>0.$
\end{lem}
\begin{proof}
Consider the polynomials $\sigma_0, s_1$ and $s_2.$
In view of the hypotheses and Proposition \ref{eps} for sufficiently small $\epsilon$ we have that
\begin{equation}\label{sc}(s_1+cs_2)\beta(c)=\sigma_0 \overline\beta(c),\end{equation}
where $\beta(c),\overline\beta(c)\in\Pi_2$ are relatively prime.

Then we have that $\beta(c)$ is a linear or conic component of $\sigma_0.$ Suppose that $\sigma_0$ has $k$ such components. By considering $k+1$ sufficiently small values of $c$ we get that there are constants $c_1$ and $c_2$ such that $\beta(c_1)=\beta(c_2)=:\beta_0$.

Then we readily obtain from \eqref{sc} that
\begin{equation}\label{s12}s_1\beta_0=\sigma_0 \beta_1\ \hbox{and}\ s_2\beta_0=\sigma_0 \beta_2,\ \hbox{where}\ \beta_1,\beta_2\in\Pi_2.\end{equation}
In the case when $\beta_0$ is relatively prime with $\beta_1$ or $\beta_2$ then it clearly divides
$\sigma_0.$ 
By denoting $\tilde\sigma_0=\sigma_0 /\beta_0\in\Pi_{k-1},$
we get \eqref{sb} from \eqref{s12}.

\noindent It remains to consider the case when $\beta_0$ is a reducible conic and has a common linear component with $\beta_1$ as well as with $\beta_2.$ Below everywhere the letter $\ell$ denotes a linear polynomial. Thus suppose that $\beta_0=\ell_0\ell_0'.$ After a cancellation with a linear polynomial in \eqref{s12} two cases are possible:
 
\noindent Case 1. $s_1\ell_0=\sigma_0 \ell_1\ \hbox{and}\ s_2\ell_0'=\sigma_0 \ell_2;$ \quad
Case 2. $s_1\ell_0=\sigma_0 \ell_1\ \hbox{and}\ s_2\ell_0=\sigma_0 \ell_2.$

In Case 1 $\beta_0=\ell_0\ell_0'$ again divides $\sigma_0$ and we get \eqref{sb}.
In Case 2 $\beta_0=\ell_0$ divides $\sigma_0$ and we get \eqref{sb}, where $\beta_0$ therefore $\beta_1$ and $\beta_2$ are linear.
 Thus \eqref{sb} is proved.

Note that if $\beta_0$ and $\beta_1$ are relatively prime then $\tilde\sigma_0$  is uniquely determined from the first two relations in \eqref{sb} as the greatest common divisor of $\sigma_0$ and $s_1.$ 

Now, consider the ``Furthermore'' statement. Assume that the pairs $\beta_0, \beta_1,$ and $\beta_0, \beta_2,$ have a common factor.
Set $\beta_0=\ell\ell_0$ and  $\beta_1=\ell\ell_1.$ Then we have that either $\beta_2=\ell\ell_2,$ or $\beta_2=\ell_0\ell_3.$ The first case reduces to the item (i).
Let us consider the second case.
It is easily seen that the polynomials $\beta_0=\ell\ell_0$ and $\beta_1+\epsilon\beta_2=\ell\ell_1+\epsilon\ell_0\ell_3$  have no common factor. 

Indeed, conversely suppose that $\ell$ is a common factor. Then the last equality implies that $\ell=\ell_0,$ or $\ell=\ell_3.$ In the first case we get that $\beta_0$ and hence, in view of \eqref{sb}, $\sigma_0$ has a duble component $\ell,$ while in the second case we get that $\beta_0=\beta_2$ and hence $\sigma_0=\sigma_2.$ 

Now conversely suppose that $\ell_0$ is a common factor.
In this case the same equality implies that $\ell_0=\ell,$ or $\ell_0=\ell_1.$ The first case was considered already, while the second case implies that $\beta_0=\beta_1$ and hence $\sigma_0=\sigma_1.$ 
   \end{proof}

Now we are in a position to present

\emph{Proof of Proposition \ref{factor3}.} Assume by way of contradiction that
any polynomial from
$S:=\hbox{Linear span}\{\sigma_1,\ldots,\sigma_6\},$
differs from $\sigma_0$ with a factor of degree at most two.
By Lemma \ref{nom}, for the polynomial $\sigma_0$ and any two polynomials from  $S,$ the relation \eqref{sb} holds.

Case 1. Assume that there is a polynomial $s_1\in S,$ say it is $s_1=\sigma_1$, for which the relation \eqref{sb} holds with $\beta_1$ being relatively prime with $\beta_0.$ Note that this evidently takes place if $\beta_0$ is linear.

Then, according to Lemma \ref{nom},  $\tilde\sigma_0$ is determined uniquelly.

Now, let us apply Lemma \ref{nom} successively with the triples of polynomials
$\sigma_0,\sigma_1,\sigma_i,\ i=2,\ldots,6.$ Then we get that
 \begin{equation*} \label{sbs} \sigma_i=\tilde \sigma_0\beta_i, \ i=0,\ldots,6,\   \hbox{where}\ \beta_i\in \Pi_2.\end{equation*}

Clearly the seven polynomials $\beta_i$ here, and consequently the seven polynomials $\sigma_i$
are linearly dependent, which contradicts our assumption.

Case 2. Assume that for any triple of polynomials $\sigma_0, s_1:=\sigma_i, s_2:=\sigma_j$ the relation \eqref{sb} holds with $\beta_0$ having a common factor with $\beta_i$ as well as with $\beta_j.$ Hence all three are of degree two.

Now, if for some triple the alternative (ii) holds then we have Case 1 with $s_1:=\sigma_i+\epsilon\sigma_j.$ Note that, in view of Proposition \ref{eps}, $s_1$ has no multiple factors if $\epsilon$ is sufficiently small.

\noindent Next, suppose that the alternative (i) holds:
$\beta_0=\ell\ell_0, \beta_i=\ell\ell_i, \beta_j=\ell\ell_j.$

This reduces to Case 1 since here \eqref{sb} holds also with linear $\beta$'s:
\begin{equation*} \label{bsss} \sigma_0=\overline \sigma_0\ell_0,\  \sigma_i=\overline \sigma_0\ell_i,\ \sigma_j=\overline \sigma_0\ell_j,\ \hbox{where}\ \overline \sigma=\tilde\sigma\ell.\qquad\Box\end{equation*} 
%\end{proof}
\section{The existence of three curves of degree $k-1$}

\begin{proposition}\label{pr3k-1}
\textit {Assume that the hypotheses of Theorem \ref{mainth} hold. Then, there are three linearly independent curves of degree $k-1$ passing through all the nodes of the set ${\mathcal X}.$}
\end{proposition}

\begin{proof}
Let  $\sigma_0,\ldots,\sigma_6,$ be the seven curves of degree $\le k$ that pass through all the nodes of the $n$-independent set ${\mathcal X}$ with $\#{\mathcal X} = d(n, k-3)+3$.

In view of Proposition \ref{nomult}, without loss of generality, assume that each of these polynomials is of exact degree $k$ and has no multiple factors.

{\bf Step 1.} Here we will prove that there is at least one curve of degree $\le k-1$ passing through all the nodes of the set ${\mathcal X}.$

We start by choosing two  nodes  $B_1,B_2\notin{\mathcal X}$ such that the following two conditions are satisfied:

\noindent $i)$ the set ${\mathcal X}\cup \{B_1,B_2\}$ is $n$-independent;

\noindent $ii)$ the line $\ell_0$ through $B_1$ and $B_2$ does not pass through any node from ${\mathcal X}.$

Let us verify that one can find such nodes. Indeed, in view of Lemma \ref{ext}, we can start by choosing some nodes $B_i=B_i',~i=1,2,$ satisfying the condition
$i)$.  Then, according to Lemma \ref{eps'}, for some positive $\epsilon$ all the nodes $B_i,\ i=1,2,$
in $\epsilon$ neighborhoods of $B_i', i=1,2,$ respectively, satisfy the condition $i)$.
Finally, from these  neighborhoods we can choose the nodes $B_i,\ i=1,2,$
satisfying the condition $ii)$ too.

Next we find one more node $B_3\in \ell_0$ such that
the set ${\mathcal X}\cup \{B_1,B_2,B_3\}$ is $n$-independent.
Indeed, if there is no such node then we obtain that
$$p\in\Pi_k,\ p|_{\Xset\cup\{B_1,B_2\}}=0 \Rightarrow p|_{\ell_0}=0.$$
Therefore $p=\ell_0 q,$ where $q\in\Pi_{k-1}$ and, in view of the condition ii), $q|_{\Xset}=0.$
Hence, if there is no $B_3$ then, according to Lemma \ref{2cor}, there are five linearly independent polynomials $p\in\Pi_k$ satisfying the condition $p|_{\Xset\cup\{B_1,B_2\}}=0.$ Therefore,
there are five linearly independent $q\in\Pi_{k-1}$ satisfying the condition $q|_{\Xset}=0.$

Next, we find successively two more nodes $B_4,B_5\in \ell_0$ such that
the set ${\mathcal X}\cup \Bset_5$ is $n$-independent,
where $\Bset_5:= \{B_1,B_2,B_3,B_4,B_5\}.$
Indeed, if one cannot find the node $B_4$ or $B_5$ then, in the same way as above, we obtain that there are four or three linearly independent polynomials $q\in\Pi_{k-1}$ satisfying the condition $q|_{\Xset}=0,$ respectively.

Then, in view of Lemma \ref{2cor}, there are two curves of degree $\le k,$  which pass through all the nodes of ${\mathcal X}\cup \Bset_5.$
Denote one of them by $\sigma_0.$ We may assume that it is of exact degree $k$
 and has no multiple factors.
We may assume also that $\ell_0$ is not a component of $\sigma_0.$ Otherwise as above, we find a desired polynomial $q.$

Now, in view of Proposition \ref{extcurve}, we enlarge the set ${\mathcal X}\cup\Bset_5$ to a maximal $n$-independent set ${\mathcal Z}\subset \sigma_0,$ by adding $d(n,k)-(d(n,k-3)+3)-5=3(n-k)+1$ nodes, i.e.,
$${\mathcal Z} = {\mathcal X}\cup\Bset_5\cup{\mathcal A},\ \hbox{where}\ \#{\mathcal A}=3(n-k)+1=[3(n-k-1)-1]+5.$$

Let us start with the description of the choice of $3(n-k-1)-1$ nodes of $\mathcal A.$
By using Proposition \ref{factor3} we find a curve $\sigma$ in the linear span of  $\sigma_i,\ i=1,\ldots,6,$ which has no multiple factors and differs from $\sigma_0$ with a factor of degree at least three:
$\sigma=\gamma r,\ \sigma_0=\gamma_0 r,$  with $d:=\deg\gamma=\deg\gamma_0\ge 3$ and $r\in \Pi_{k-d}.$ We have that $\gamma_0$ and $\sigma$ are relatively prime. 
%In the same way as above we may assume also that $\ell_0$ is not a component of $\sigma.$

Below we are using Theorem \ref{intn} with respect to the curve $\Cset:=\gamma_0.$ Choose a point $O\notin \gamma_0\cup\sigma.$ Since $O\notin \sigma_0$ no line through the point $O$ will be a component of $\sigma_0.$ Consider a line $\ell_1$ through $O$ which intersects $\Cset$ at distinct points not belonging to $\ell_0\cup\sigma.$  Let  $A_1,A_2$ and $A_3,$ be three of those intersection points.  By using a continuity argument we may assume that the lines $\ell_i,\ i=2,\ldots,n-k-1,$ pass through $O$ and are enough close to $\ell_1$ so that each of them intersects $\Cset$ at distinct points, which do not belong to $\ell_0\cup\sigma.$ We assume also that $\ell_i\cap(\Xset\cup\Bset_5)=\emptyset,\ i=1,\ldots,n-k-1.$    As in the case of the line $\ell_1$ let $A_{3i-2},A_{3i-1}$ and $A_{3i},$ be three of those intersection points belonging to $\gamma_0\cap\ell_i,\ i=2,\ldots,n-k+1.$ Finally, let us dismiss an intersection point, say $A_1,$ and denote the desired set of the remaining  $3(n-k-1)-1$ intersection nodes 
$\{A_i\}$ by $\Aset(-1).$

Let us prove that the set $\Yset:=\Xset\cup\Bset_5\cup\Aset(-1)$ is $n$-independent.

We have that the set $\Aset(-1)$ is a subset of Berzolari-Radon construction of degree $n-k-1.$
Hence it is $(n-k-1)$-independent. Now suppose that $p^\star_{A,\Aset(-1)}$ is a fundamental polynomial of a node $A\in\Aset(-1)$ of degree $n-k-1.$ Then the polynomial
$\sigma\ell_0p^\star_{A,\Aset(-1)}$ is an $n$-fundamental polynomial of the node $A$ for the set $\Yset.$ Here we use the fact that no node from $\Aset(-1)$ belongs to $\ell_0$ or $\sigma.$ Thus, according to Lemma \ref{XA}, the set $\Yset$ is $n$-independent.

Finally, in view of Proposition \ref{extcurve}, we enlarge the set $\Yset\subset\sigma_0$ with a set $\Aset_5$ of the last $5$ nodes to a maximal $n$-independent set $\Zset\subset\sigma_0.$
Thus we have that  $\Zset:=\Yset\cup\Aset_5$ and $\Aset=\Aset(-1)\cup\Aset_5.$

Now suppose that $\sigma^*\in\Pi_k$ vanishes on ${\mathcal X}$ and $\Aset_5.$
 According to Lemma \ref{2cor} there are $2=7-5$ such polynomials. Hence we may assume that $\sigma^*\neq\sigma_0.$ Then consider the polynomial $p=\sigma^*\ell_0\ell_1\cdots\ell_{n-k-1}.$ We have that $p\in\Pi_n$ vanishes on the maximal $n$-independent set $\Zset\subset\sigma_0.$ Therefore,  we have that
$p=\sigma^*\ell_0\ell_1\cdots\ell_{n-k-1}=\sigma_0s,\ \hbox{where}\ s\in\Pi_{n-k}.$

The lines $\ell_i,\ i=1,\ldots,n-k-1,$ are not components of $\sigma_0$ since they  pass through $O\notin\sigma_0.$ Therefore, they are  components of the polynomial $s.$ Thus we obtain
$$\sigma^*\ell_{0}=\sigma_0\ell,\ \hbox{where}\ \ell\in\Pi_{1}.$$
Since $\sigma^*\neq\sigma_0$ therefore $\ell_0\neq \ell.$ Whence $\ell_{0}$ is a component of $\sigma_0:$
$\sigma_0=\ell_{0}q_0,\ \hbox{where}\ q_0\in\Pi_{k-1}.$
As above we get that $q_0$ vanishes on $\Xset.$

{\bf Step 2.} Here we will prove that there are three linearly independent curves of degree $\le k-1$ passing through all the nodes of the set ${\mathcal X}.$

We find a line $\ell_0$ and  collinear nodes $B_1,\ldots,B_4\in\ell_0,$ in the same way as in the Step 1,  such that $\ell_0\cap\Xset=\emptyset$ and
the set ${\mathcal X}\cup \Bset_4$ is $n$-independent,
where $\Bset_4:= \{B_1,B_2,B_3,B_4\}.$

Next, in view of Proposition \ref{2cor}, there are three linearly independent curves of degree at most $k,$  which pass through all the nodes of the set ${\mathcal X}\cup \Bset_4.$
Denote these curves by $\sigma_0, \sigma_0', \sigma_0''.$
If a curve here, say $\sigma_0,$ is of degree $\le k-1$ and has no multiple components then
instead of given triple of curves we consider the curves
$\ell_1\sigma_0, \ell_2\sigma_0,\ell_3 \sigma_0,$ where the lines $\ell_i$ are chosen such that these three curves are linearly independent and have no multiple factors.

Next, if a curve $\sigma_0, \sigma_0',\sigma_0'',$ has a multiple factor then by throwing away the excessed factor we are in the situation considered in the previous paragraph.
Hence, we may consider only the case when each of theses three polynomials is of exact degree $k$ and has no multiple components.

Now consider the curve $\sigma_0.$
In view of Proposition \ref{extcurve}, we enlarge the set ${\mathcal X}\cup \Bset_4$ to a maximal $n$-independent set ${\mathcal Z}\subset \sigma_0,$ by adding $d(n,k)-(d(n,k-3)+3)-4=3(n-k)+2$ nodes, i.e.,
$${\mathcal Z} = {\mathcal X}\cup\Bset_4\cup{\mathcal A},\ \hbox{where}\ \#{\mathcal A}=3(n-k)+2=[3(n-k-1)-1]+1+5.$$

We find the set of $3(n-k-1)-1$ points from $\mathcal A$ in the same way as in Step 1  and denote it again by $\Aset(-1).$
Then, in the same way as in Step 1, we prove the independence of the set $\Yset:={\mathcal X}\cup\Bset_4\cup{\mathcal A}(-1).$

Next, in view of Theorem \ref{intn},
we choose a node  $\tilde A_1\in\ell_1$ such that $\tilde A_1\in\sigma_0\setminus q_0,$
where  $q_0$ is the polynomial of degree $\le k-1$ vanishing on $\Xset,$ found in Step 1.
Note that the line $\ell_1$ is not a component of $q_0$ since $\ell_1\cap\Xset=\emptyset.$

Then consider the case when $\tilde A_1\in\Aset(-1),$ i.e., $\tilde A_1$ coincides with one of the nodes $A_2,A_3\in\Aset(-1)\cap\ell_1,$ say $\tilde A_1=A_2.$ In this case instead of $\Aset(-1)$ we would start with the set $\Aset(-1)'=\Aset(-1)\cup\{A_1\}\setminus \{A_2\}$ and we will have already that $\tilde A_1\notin\Aset(-1)'.$

Since $\ell_0$ is not a component of $\sigma_0$ therefore the set $F:=\ell_0\cap \sigma_0$ is a finite set and we could suppose beforehand that $\ell_1\cap F=\emptyset.$ This will ensure that $\tilde A_1\notin\ell_0.$  Also we have that $\tilde A_1\neq O$ since $O\notin \sigma_0.$

Now let us prove the independence of the set $\tilde{\mathcal Y}:={\mathcal Y}\cup\{ \tilde A_1\}.$
For this end, in view of Lemma \ref{XA}, it suffices to find a fundamental polynomial of the node $\tilde A_1$ with respect to the set $\tilde{\mathcal Y}.$
We readily verify that $p^\star_{\tilde A_1, \tilde\Yset}=q_0\ell_0\ell_2\cdots\ell_{n-k-1}\ell'\ell'',$ where $\ell'$ and $\ell''$ are lines different from $\ell_1$ and pass through the nodes $A_2$ and $A_3,$ respectively. 

Finally, according to Proposition \ref{extcurve}, let us enlarge the set $\tilde\Yset\subset\sigma_0$  with the set of last $5$ nodes, denoted by $\Aset_5,$ to a maximal $n$-independent set.
Thus the set $\Zset:=\tilde\Yset\cup\Aset_5$ is a maximal $n$-independent set in $\sigma_0.$

Now suppose that $\sigma^*\in\Pi_k$ vanishes on ${\mathcal X}$ and the $5$ nodes of $\Aset_5.$
 According to Lemma \ref{2cor} there are at least two such polynomials. Hence we may assume that $\sigma^*\neq\sigma_0.$ Then consider the polynomial   $p=\sigma^*\ell_0\ell_1\cdots\ell_{n-k-1}.$ We have that $p\in\Pi_n$ and $p$ vanishes on the
 node set $\Zset,$ which is a maximal $n$-independent set in the curve $\sigma_0.$ Therefore,  we have that
 
 $\qquad\qquad p=\sigma^*\ell_0\ell_1\cdots\ell_{n-k-1}=\sigma_0s,\ \hbox{where}\ s\in\Pi_{n-k}.$

The lines $\ell_i,\ i=1,\ldots,n-k-1,$ are not components of $\sigma_0.$ Therefore, they are  components of the polynomial $s.$ Thus we get that
$\sigma^*\ell_{0}=\sigma_0\ell,\ \hbox{where}\ \ell\in\Pi_{1}.$
Since $\sigma^*\neq \sigma_0$ therefore $\ell_0\neq \ell.$ Hence $\ell_{0}$ is a component of $\sigma_0:$
$$\sigma_0=\ell_{0}q_{k-1},\ \hbox{where}\ q_{k-1}\in\Pi_{k-1}.$$
In the same way for the curves $\sigma_0'$ and $\sigma_0''$ we get
$\sigma_0'=\ell_{0}q_{k-1}',\ \hbox{where}\ q_{k-1}'\in\Pi_{k-1},\ $ and
$\ \sigma_0''=\ell_{0}q_{k-1}'',\ \hbox{where}\ q_{k-1}''\in\Pi_{k-1}.$

Obviously the  curves $q_{k-1}, q_{k-1}', q_{k-1}'',$ are linearly independent.
\end{proof}

\section{\label{proofs} Proofs of Theorems \ref{mainth} and \ref{hkv}}

Let us start with
\begin{proof}[Proof of Theorem \ref{hkv}.]
Assume by way of contradiction that there are four curves passing through all the nodes of the set $\Xset.$ Then, according to Theorem \ref{hk}, all the nodes of $\Xset$ but three belong to a maximal curve $\mu$ of degree $k-2.$
The curve $\mu$ is maximal and the remaining three nodes of ${\mathcal X},$ denoted by $A, B$ and $C,$ are outside of it: $A,B,C\notin \mu.$ Hence we have that
$${\mathcal P}_{k,{\mathcal X}}=\left\{{p : p\in \Pi_{k}}, p_{\mathcal X}=0\right\}= \left\{{q\mu : q\in \Pi_{2}}, q(A)=q(B)=q(C)=0\right\}. $$
Thus we get readily that
$\dim{\mathcal P}_{k,{\mathcal X}}=\dim\left\{{q\in \Pi_{2}}: q(A)=q(B)=q(C)=0\right\}$ $=\dim{\mathcal P}_{2, \{A,B,C\}}=6-3=3,$
which contradicts our assumption.
Note that in the last equality we use Proposition \ref{PnX} and the fact that any three nodes are $2$-independent.

Now, let us verify the part ``if''. By assuming that there is a curve $\sigma$ of degree $k-1$ passing through the nodes of $\Xset$ we find readily three linearly independent curves of degree $\le k:\ $ $\sigma, x\sigma,
y\sigma,$ passing through $\Xset.$  While if we assume that all the nodes of  ${\mathcal X}$ but three lie in a curve $\mu$ of degree $k-2$ then above evaluation shows that $\dim{\mathcal P}_{k,{\mathcal X}}=3.$

Finally, let us verify the part ``only if''.
Denote the three  curves passing through all the nodes of ${\mathcal X}$ by $\sigma_0, \sigma_0', \sigma_0''.$
If one of them is of degree $k-1$ then the conclusion of Theorem is satisfied and we are done. Thus, we may assume that each curve is of degree $k$ and has no multiple components.
Now consider the curve $\sigma_0.$

By using Proposition \ref{extcurve} let us enlarge the set ${\mathcal X}$ to a maximal
$n$-independent set ${\mathcal Z}\subset\sigma_0.$ Since
$\#{\mathcal Z} = d(n,k)$, we need to add  a set of $d(n, k)-(d(n, k-2)+3) = 2(n-k)+2$
nodes, denoted by $$\mathcal A:=\{A_1,\ldots,A_{2(n-k)+2}\}.$$
Thus we have that ${\mathcal Z}:={\mathcal X}\cup \mathcal A.$
In view of  Lemma \ref{lem44}, $i)$, we require that each node of $\mathcal A$ may belong only to one component of the curve $\sigma_0.$

{\bf Case 1,} $n=k+2,\ \mathcal A:=\{A_1,\ldots,A_6\}.$

Consider $5$ nodes from $\Aset$ and a conic $\beta^*$ passing through them. Denote the sixth node by $A^*.$
We have three polynomials from $\Pi_{k}$ vanishing on $\Xset.$
By using Lemma \ref{2cor} we get two linearly independent curves  of degree at most $k,$  that pass through all
the nodes of ${\mathcal X}$ and the node $A^*\in\Aset.$
Thus we may consider a such  curve $\sigma^*\in\Pi_k$ by assuming that $\sigma^*\neq\sigma_0.$
Now, notice that the  polynomial
$
\sigma^* \, \beta^*
$
of degree $n$ vanishes at all the nodes of ${\mathcal Z}\subset\sigma_0.$ Consequently, according to Proposition \ref{maxcurve}, $\sigma_0$ divides this polynomial:
\begin{equation}\label{sigmalq}
\sigma^*\, \beta^*= \sigma_0 \, \beta, \quad \beta \in\Pi_{2}.
\end{equation}

We have that $\beta^*\neq\beta$ since $\sigma^*\neq \sigma_0.$
Hence if  $\beta^*$ is irreducible then it divides $\sigma_0.$
Now suppose that $\beta^*$  is reducible: $\beta^*=\ell_1\ell_2,$
where $\ell_i\in\Pi_1.$
Then we have that both lines $\ell_1,\ell_2,$ cannot divide $\beta$, hence either $\ell_1\ell_2$ or only one of them is a component of $\sigma_0.$

Let us consider the latter case. Suppose that  the line $\ell_1$ is a component of $\sigma_0$ and $\ell_2$ is a component of $\beta.$ Then we get from \eqref{sigmalq}
that
\begin{equation}\label{sigmalq1}\sigma^*\, \ell_1= \sigma_0 \, \ell, \ \hbox{where}\ \ell\in\Pi_1.\end{equation}
Now, we have that $\sigma_0=\ell_1 q,$ where $\deg q={k-1}.$
Then we get from \eqref{sigmalq1}  that $\sigma^*=\ell q.$ From the last two equalities we conclude that  $\Xset\subset q \cup \{E\},$
where $E=\ell_1\cap\ell.$

Therefore all the nodes of $\Xset,$ except possibly $E,$ belong to the curve $q.$ Here $q$ is a component of $\sigma_0$ of degree $k-1$ and $E$ belongs to its line component $\ell_1.$

We briefly express the above conditions by saying that the line component  $\ell_1$ of $\sigma_0$ satisfies $(-1)$-node condition for $\Xset.$

At the end we will see that if this property holds for all three given curves $\sigma_0, \sigma_0', \sigma_0'',$ then we can readily complete the proof of Theorem.

Therefore, from now on we may assume that the equality \eqref{sigmalq} implies that $\deg\beta^*=2$ and  $\beta^*$  is a component of $\sigma_0.$  Thus we obtain also that $\beta^*$ is determined uniquely by the  $5$ nodes from $\Aset.$

Next, we are going to prove that there is a conic passing through all the six nodes of $\Aset.$ Assume conversely that there is no such conic. Denote by $\beta_i$ the conic passing through the five nodes of   $\Aset_6\setminus\{A_i\},\ i=1,2.$

We have that these two conics  are different
components of $\sigma_0.$ First assume that one of these two conics, say, $\beta_1$, is irreducible. Then consider a common node of $\beta_1$ and $\beta_2$, say, $A_3.$ It is easily seen that $A_3$ belongs to two different components of $\sigma_0,$ which contradicts our assumption. Indeed, one is $\beta_1$ and another is $\beta_2$ if it is irreducible or a line component of $\beta_2$ if it is reducible.

 Now, assume that both $\beta_1$ and $\beta_2$ are reducible:
$\beta_1=\ell_1\ell_1', \quad  \beta_2=\ell_2\ell_2'.$
Wihout loss of generality assume that
\begin{equation}\label{llll}\ell_1\neq \ell_2, \quad \ell_1\neq \ell_2'.\end{equation}

We have that $\ell_1$ passes through at least one of the common nodes $A_3,\ldots,A_6,$ say $A_3.$ Then $A_3$ belongs either to $\ell_2$ or to $\ell_2'.$ In both cases, in view of \eqref{llll}, we have that $A_3$ belongs to two different line components of $\sigma_0,$ which is a contradiction.
Thus we proved that $\Aset\subset\beta_0,$ where $\beta_0\in\Pi_2.$

Next let us show that $\beta_0$ divides $\sigma_0.$
Consider a polynomial $\sigma\in\Pi_{k}$ that vanishes on $\Xset$ and $\sigma\neq\sigma_0.$
Notice that the following polynomial
$
\sigma \, \beta_0
$
of degree $k+2\le n$ vanishes at all the $d(n,k)$ nodes of ${\mathcal Z}\subset\sigma_0.$ Consequently, according to Proposition \ref{maxcurve}, $\sigma_0$ divides this polynomial:
\begin{equation}\label{sigmalqlq}
\sigma\, \beta_0
= \sigma_0 \beta, \quad \beta \in\Pi_{2}.
\end{equation}
This is a type \eqref{sigmalq} equality which, as we mentioned above, implies that $\deg\beta_0=2$ and $\beta_0$ is a component of $\sigma_0,$ i.e.,
$\sigma
= \beta_0 q, \quad q \in\Pi_{k-2}.$
We conclude also that $\beta_0$ is uniquely determined by any $5$ nodes from $\Aset.$

Thus to enlarge the set $\Xset\subset\sigma_0$ to a maximal $n$-independent set $\Zset=\Xset\cup\Aset$ we  have to add all the six nodes of $\Aset$  to the conic $\beta_0.$ Let us verify that the added nodes cannot belong to the component $q.$ Indeed, suppose conversely that a node belongs to $\beta_0\cap q.$ Then, in view of Lemma \ref{lem44}, we can move the node to $q\setminus\beta_0$ such that the resulted set is also $n$-independent. This is a contradiction, since now the six nodes do not belong to a conic. Indeed, the five nodes determine a unique conic and the sixth node is outside of it.
Thus the factor $r\in\Pi_{k-2}$ to which one can not add a new independent node is merely maximal with respect to $\Xset.$ This means that $r$ passes through exactly $d(n,k)$ nodes of $\Xset.$

{\bf Case 2,} $n\ge k+3.$

Consider a subset of $\mathcal A$ of cardinality $4$ and denote it by $\Aset_4.$
Denote also by $\bar \Aset:=\Aset\setminus\Aset_4.$ We have that $\#\bar\Aset=2(n-k)-2.$

There are three linearly independent polynomials $\sigma_0, \sigma_0', \sigma_0''\in\Pi_k,$ vanishing on $\Xset.$
Now suppose that $\sigma^*\in\Pi_k$ vanishes on ${\mathcal X}$ and at an arbitrary node $A^*\in\bar\Aset,$ which will be specified below.
 According to Lemma \ref{2cor} there are two such polynomials. Hence we may assume that $\sigma^*\neq\sigma_0.$ 
 We call the node $A^*$ associated with $\sigma^*.$

We associate another node $A'\in\bar\Aset$ with the set $\Aset_4$ and denote by $\beta'$ a conic that passes through $A'$ and the four nodes of $\Aset_4.$ 

For any line component $\ell$ of $\sigma_0$ denote by $r_\ell\in\Pi_{k-1}$ for which
\begin{equation}\label{rell}\sigma_0=\ell r_\ell.\end{equation} 
Assume that a line component $\ell$ of the curve $\sigma_0,$  passes through exactly $m$ nodes from $\mathcal X,$ at which $r_\ell$ does not vanish.
Then we obtain from \eqref{rell} that $r_\ell\in\Pi_{k-1}$ vanishes at the all nodes of the set $\Xset$ except $m$ nodes, which belong to $\ell.$

Note that if for a line $\ell$ we have that $m\le 1,$  then  the line component $\ell$ of $\sigma_0$ satisfies the $(-1)$-node condition for $\Xset.$

Therefore we may suppose that $m\ge 2$ for all lines $\ell,$ meaning that the following condition takes place:

({\bf C}) Any line component of the curve $\sigma_0,$ passes through at least two nodes from $\mathcal X,$ at which $r_\ell$ does not vanish.

Later, in Section \ref{div}, by using the condition (C), we divide the set of nodes $\bar\Aset$ into $n-k-2$ pairs such that the lines
$\ell_1,\ldots,\ell_{n-k-2},$ through them, respectively, are not components of $\sigma_0.$ The remaining two nodes denoted by  $A^*$ and $A',$ are associated with the curve $\sigma^*$ and $\Aset_4,$ respectively.

Now, let us continue the proof by assuming that the above-described  division of $\bar{\mathcal A}$ is established.

Notice that the following polynomial
$
\sigma^* \, \beta' \, \ell_1 \dots \,
\ell_{n-k-2}
$
of degree $n$ vanishes at all the $d(n,k)$ nodes of ${\mathcal Z}\subset\sigma_0.$ Consequently, according to Proposition \ref{maxcurve}, $\sigma_0$ divides this polynomial:
\begin{equation}\label{sigmalqq}
\sigma^*\, \beta'\, \ell_1 \, \dots \, \ell_{n-k-2}
= \sigma_0 \, r, \quad r \in\Pi_{n-k}.
\end{equation}
The distinct lines  $\ell_1, \dots , \ell_{n-k-2}$ do not divide the polynomial $\sigma_0 \in\Pi_{k}$, therefore, all they have
to divide $r .$  Hence, we get from \eqref{sigmalqq} that
$\sigma^* \, \beta' = \sigma_0 \, \beta, \ \hbox{where}\ \beta\in\Pi_2.$
Then, we have that  $\beta'\neq\beta$ since $\sigma^*\neq\sigma_0.$
Now, in the same way as in Case 1 we obtain that
$
\sigma_0= \beta' q \ \hbox{where}\ q\in\Pi_{k-2}.
$

Next, we are going to prove that there is a conic passing through all the nodes of $\Aset.$
Assume by way of contradiction that there is no such conic. Then, in view of Proposition \ref{hm}, we have that there is a set of six nodes, say $\Aset_6:=\{A_1,\ldots,A_6\}\subset\Aset,$ that does not lie in a conic.

Now, let us choose three noncollinear nodes in $\Aset_6,$ say  $A_1,A_2,A_3,$ and consider the following sets of four nodes:
\begin{equation*}\label{tt}A_1,A_2,A_3,A_4;\quad A_1,A_2,A_3,A_5;\quad A_1,A_2,A_3,A_6.\end{equation*}

Then, consider these three sets with the respective associated nodes:
\begin{equation}\label{tta}A_1,A_2,A_3,A_4,A';\quad A_1,A_2,A_3,A_5,A'';\quad A_1,A_2,A_3,A_6,A'''.\end{equation}
We have that the three conics through these sets are components of $\sigma_0.$
Since $\Aset_6$ does not lie in a conic we obtain that these three conics cannot coincide. Hence there are two different conics, say the conics $\beta'$ and $\beta'',$ passing through the first two sets in \eqref{tta}, respectively.

 First assume that one of these two conics, say, $\beta'$, is irreducible. Then consider a common node, say, $A_1.$ It is easily seen that $A_1$ belongs to two different components of $\sigma_0,$ which contradicts our assumption. Indeed, one is $\beta'$ and another is $\beta'',$ if it is irreducible too, or a line component of $\beta'',$ if it is reducible.

 Next, assume that both $\beta'$ and $\beta''$ are reducible: $\beta'=\ell_1\ell_1', \quad  \beta''=\ell_2\ell_2'.$
Without loss of generality assume that
\begin{equation}\label{llll'}\ell_1\neq \ell_2, \quad \ell_1\neq \ell_2'.\end{equation}
Note that $\ell_1$ passes through at least one of the common nodes $A_1,A_2,A_3,$ say $A_1.$ Indeed, if $\ell_1$ passes through only $A'$ and $A_4$ then we obtain that $\ell_1'$ passes through the three noncolinear  nodes $A_1,A_2,A_3.$  Now, we have that $A_1$ belongs either to $\ell_2$ or $\ell_2'.$ In both cases, in view of \eqref{llll'}, we have that $A_1$ belongs to two different line components of $\sigma_0,$ which is a contradiction.

Thus we proved that $\Aset\subset\beta_0,$ where $\beta_0\in\Pi_2.$
Next, in the same way as in Case 1, we show that $\beta_0$ divides $\sigma_0:$
$\sigma_0
= \beta_0 q, \quad q \in\Pi_{k-2}.$
Also we have that $\beta_0$ is uniquely determined by the nodes of $\Aset\setminus\{A\},\ \forall A\in\Aset.$

Indeed, assume conversely that $\beta_0$ is not uniquely determined by the nodes from $\Aset\setminus\{A_0\},$ where $A_0\in\Aset.$ Therefore there are infinitely many conics $\beta_0$ passing through the nodes of $\Aset\setminus\{A_0\}.$
Recall that for (any) $A_0$ one can find a curve, denoted by $\sigma^*,$ of degree at most $k,$  that passes through all
the nodes of ${\mathcal X}$ and is different from $\sigma_0.$
Then, as in Case 1, we readily get $\sigma^*\beta_0=\sigma_0 \beta,$ where $\beta\in\Pi_2.$ This implies that  $\beta_0$ is a component of $\sigma_0.$ Therefore $\sigma_0$ has infinitely many components, which is a contradiction.

Thus to enlarge the set $\Xset\subset\sigma_0$ to a maximal $n$-independent set $\Zset=\Xset\cup\Aset$ we  have to add all the nodes of $\Aset$  to the conic $\beta_0.$ Let us verify that the added nodes do not belong to the component $q.$ Suppose conversely that a node $A_0\in\Aset$ belongs to $\beta_0\cap q.$ Then, in view of Lemma \ref{lem44}, let us move  $A_0$ to $q\setminus\beta_0$ such that the resulted set $\Aset$ remains $n$-independent. This is a contradiction, since now the nodes of $\Aset$ do not belong to a conic. Indeed, the nodes $\Aset\setminus\{A_0\}$ determine a unique conic and the moved node is outside of it.
Therfore, the factor $r\in\Pi_{k-2}$ to which one cannot add a new independent node is merely maximal with respect to $\Xset.$ Hence, $r$ passes through exactly $d(n,k)$ nodes of $\Xset.$

At the end, before establishing the division of the set $\bar\Aset,$ it remains to consider the case when the division may be not possible for all three curves $\sigma_0, \sigma_0', \sigma_0'',$ i.e., the case when the condition (C) does not hold.
Then, we obtain three curves $q, q',q'',$ which are components of degree $k-1$ of the curves $\sigma_0, \sigma_0', \sigma_0'',$ respectively,  passing through all the nodes of $\Xset$ except possibly one. 

Assume that $q, q',q'',$ pass through all the nodes of $\Xset$ except $E, E',E'',$ respectively.
First assume that two of these three nodes are different, say $E\neq E'.$
We have that $q$ and $q'$ pass through all the nodes of the set $\Yset:=\Xset\setminus \{E,E'\},\ \#\Yset=d(n,k-3)+1.$
If $q=q'$ then  we have that $E=E',$ contradicting our assumption. If $q\neq q'$ then, according to Theorem \ref{ht2}, all the nodes of $\Yset$ except one belong to a (maximal) curve $\mu$ of degree $k-2.$
Thus all the nodes of $\Xset$ except three belong to $\mu.$

It remains to consider the case $E = E'= E''.$
Then we have that $q,q',q'',$ pass through all the nodes of the set $\Yset:=\Xset\setminus \{E\},\ \#\Yset=d(n,k-2)+2.$ We get from Theorem \ref{ht} that $q=q'=q''=:q.$

Next, in view of the condition (C), we get that $\sigma=\ell  q,\  \sigma'=\ell'  q,\  \sigma''=\ell'',$ where
$\ell,\ell'\ell''\in\Pi_1.$ This contradicts the linear independence of $\sigma,\sigma'\sigma'',$ since we have that $E\in \ell\cap\ell'\cap\ell''.$
\end{proof}
{\bf Proof of Theorem \ref{mainth}.}
It is easily seen that Theorem \ref{mainth} follows from Proposition \ref{pr3k-1}, Theorem \ref{hkv} and Lemma \ref{mainl}. $\hfill\Box$

\subsection{\label{div}The division of the set $\bar\Aset$}
Next let us establish the above mentioned division of the node set  $\bar \Aset:=\Aset\setminus\Aset_4$ in the case $n\ge k+3.$ Note that this is the case when we need the division.

Recall that each node of $\mathcal A$ belongs only to one component of the curve $\sigma_0.$ 
By using induction on $m$ one can prove easily the following
\begin{lem}[Proof of Th. 3, \cite{HK}]\label{lemA0} Suppose that a finite set of lines ${\mathcal L}$ and  $2m$ nodes lying in these lines are given. Suppose also that no node is an intersetion point of two lines. Then one can divide the node set into $m$ pairs such that no pair belongs to the same line from ${\mathcal L}$ if and only if each line from ${\mathcal L}$ contains no more than $m$ nodes.
\end{lem}
Thus the above mentioned division of the node set $\bar{\mathcal A}$ into $n-k-2$ pairs is possible if and only if no $n-k-1$ nodes of $\bar{\mathcal A}_0:=\mathcal A\setminus \{A^*,A'\}$ are located in a line component of $\sigma_0,$ where the nodes $A^*$ and $A'$ are the nodes associated with the curve $\sigma^*$ and $\Aset_4,$ respectively. Observe also that we may associate any  two nodes $A^*$ and $A'$ of $\mathcal A$ with $\sigma^*$ and $\Aset_4,$

Now notice that, in view of $\#\bar\Aset=2(n-k-1),$ there can be at most two
undesirable line components for the set $\bar{\mathcal A},$ i.e., lines containing at least $n-k-1$ nodes from it. In this case a node from each line we assign as associated and leave in the two lines $\le n-k-2$ nodes.

Then assume that we have one
undesirable line component for the set $\bar{\mathcal A},$ containing  $\le n-k$ nodes from it.  In this case two nodes from this line we nominate as associated and leave in the line $\le n-k-2$ nodes.

Finally consider the case of one undesirable line component $\ell$ of $\sigma_0$ with $m \ge n-k+1$ nodes. We have that
$$\sigma_0=\ell r_\ell,\ \hbox{where}\ r_\ell\in\Pi_{k-1}.$$
 Now we are going to move $m-n+k$ nodes, one by one, from $\ell$ to the other component $r_\ell$ such that the set ${\mathcal Z}:={\mathcal X}\cup \mathcal A$ remains $n$-independent. Again, in view of Lemma \ref{lem44}, i), we require that each moved node belongs only to one component of the curve $\sigma_0.$

 To establish each described movement, in view of Lemma \ref{lem44}, ii), it suffices to prove that during this process each node $A\in\ell\cap \mathcal A,$ has no $n$-fundamental polynomial for which the curve $r_\ell$ is a component.  Suppose conversely that 
 \begin{equation}\label{pstar}
 p^\star_A=r_\ell s, \ s\in \Pi_{n-k+1}.
 \end{equation} Now, we have that $s$ vanishes at $\ge n-k$ nodes in $\ell\cap \mathcal A\setminus\{A\}.$ Indeed, the nodes of the set $\Aset$ in the line $\ell$ do not belong to another component. Therefore, $r_\ell$ does not vanish at these nodes and hence, in view of \eqref{pstar}, $s$ vanishes. According to the condition (C) $r_\ell$ does not vanish also at least at two nodes from $\ell\cap \mathcal X,$ and hence $s$ vanishes  there too. Thus the number of zeroes of $s$ in the line $\ell$ is greater or equal to $n-k+2$ and $s$ together with $p^\star_A$ vanishes at the whole line $\ell,$ including the node $A,$
which is a contradiction.

It remains to note that there will be no more undesirable lines, except $\ell,$ in the resulted set $\mathcal A,$ after the described movement of the nodes, since we finish by keeping exactly
$n-k$ nodes in $\ell\cap \mathcal A$ and outside of it there are only $n-k-2$ nodes.

\section{An application to bivariate  interpolation}

A $GC_n$ set ${\mathcal X}$ in the plane is an $n$-poised set of nodes, where the fundamental polynomial of each node is a product of $n$ linear factors.
The Gasca--Maeztu conjecture states that any $GC_n$-set possesses a subset of $n+1$ collinear nodes.

Recall that a node $A\in {\mathcal X}$ uses a line $\ell$ means that $\ell$ is a factor of the fundamental polynomial,  i.e., $p^\star_{A} = \ell r$ for some $r \in \Pi_{n-1}.$

It was proved by Carnicer and Gasca  in \cite{CG}, that any line passing through exactly $2$ nodes of a  $GC_n$-set ${\mathcal X}$ can be used at most by one
node from ${\mathcal X}.$
Next, it was proved in \cite{HT} that any used  line passing through exactly $3$ nodes of a $GC_n$-set ${\mathcal X}$ can be used either by exactly one
or three nodes from ${\mathcal X}.$
In \cite{HK} was proved that a line $\ell$ passing through exactly $4$ nodes can be used at most  by six nodes from ${\mathcal X}.$  Moreover, if it is used by at least four nodes then it is used by exactly six nodes from ${\mathcal X}.$

Below we consider the case of lines passing through exactly $5$ nodes.
\begin{cor}
Let ${\mathcal X}$ be an $n$-poised set of nodes and  $\ell$ be a line which passes through exactly $5$ nodes. Then $\ell$ can be used at most  by ten nodes from ${\mathcal X}.$ 
Moreover, if $\ell$ is used by at least seven nodes from ${\mathcal X}$ then it is used by exactly ten nodes from ${\mathcal X}.$ Furthermore, if it is
used by ten nodes, then they form a $3$-poised set.  In the latter case, if ${\mathcal X}$ is a $GC_n$ set then the ten nodes form a $GC_3$ set too.
\end{cor}

\begin{proof}
Assume that $\ell\cap{\mathcal X}=\{A_1,\ldots,A_5\}=:\mathcal A.$ Assume also that the seven nodes in $\mathcal B:=\{B_1,\ldots,B_7\}\in {\mathcal X}$ use the line
$\ell:$
$p_{B_i}^{\star}=\ell \, q_i,\ i=1,\ldots,7,$
where $q_i \in \Pi_{n-1}.$

The polynomials $q_1,\ldots, q_7,$ vanish at $N-12$ nodes of the set ${\mathcal X}':={\mathcal X}\setminus(\mathcal A\cup\mathcal B).$ Hence through these $N-12=d(n, n-4)+3$ nodes pass seven linearly independent curves of degree $n-1.$
By Theorem \ref{mainth} there exists a maximal curve
$\mu$ of degree $n-4$ passing through $N-15$ nodes of ${\mathcal X}'$ and the remaining three nodes denoted by $C_1,C_2,C_3,$ are outside of it.
Now, according to Proposition \ref{maxcor}, the nodes $C_1,C_2,C_3,$ use $\mu:$
$p_{C_i}^{\star} = \mu r_i,\ r_i \in \Pi_4,\ i=1,2,3.$

These polynomials $r_i$ have to vanish at the five nodes of $\mathcal A\subset\ell.$
Hence $r_i=\ell\gamma_i,\ i=1,2,3,$ with $\gamma_i\in \Pi_3.$
Therefore, the nodes $C_1,C_2,C_3,$ use the line ${\ell}:$
$p_{C_i}^{\star} = \mu{\ell}\gamma_i,\ i=1,2,3.$
Hence if seven nodes in $\mathcal B\subset{\mathcal X}$ use the line ${\ell}$ then there exist three
more nodes $C_1,C_2,C_3\in{\mathcal X}$ using it and all the nodes of ${\mathcal Y}:={\mathcal X}\setminus (\mathcal A\cup\mathcal B\cup\{C_1,C_2,C_3\})$ lie in a
maximal curve $\mu$ of degree $n-4:$
\begin{equation}\label{D}{\mathcal Y}\subset\mu.\end{equation}

Next, let us show that there is no
eleventh node using ${\ell}$. Assume conversely that except of the ten nodes in $\mathcal S:=\{B_1,\ldots,B_7, C_1,C_2,C_3\},$ there is an eleventh node $D$ using ${\ell}.$ Of course we have that
$D\in{\mathcal Y}.$

Then we have that the seven nodes $B_1,\ldots,B_6$ and $D$ are using $\ell$ therefore, as was proved above, there exist three
more nodes $E_1,E_2,E_3\in{\mathcal X}$  (which may coincide or not with $B_7$ or $C_1,C_2,C_3$) using it and all the nodes of ${\mathcal Y}':={\mathcal X}\setminus (\mathcal A\cup
\{B_1,\ldots,B_6,D, E_1,E_2,E_3\})$ lie in a
maximal curve $\mu'$ of degree $n-4.$

We have also that
\begin{equation}\label{E}
p_{D}^{\star} = \mu' q', \ q' \in \Pi_4.
\end{equation}

Now, notice that both the curves $\mu$ and $\mu'$ pass through all the nodes of the set
${\mathcal Z}:={\mathcal X}\setminus (\mathcal A\cup\mathcal B\cup\{C_1,C_2,C_3,D,E_1,E_2,E_3\})$ with $\#{\mathcal Z}\ge N-19.$

Then, we get from Theorem \ref{ht}, with $k=n-5$, that $N-19 = d(n, n-5)+2$ nodes determine
the curve of degree $n-4$ passing through them uniquely. Thus $\mu$ and $\mu'$ coincide.

Therefore, in view of \eqref{D} and \eqref{E}, $p_{D}^{\star}$ vanishes at all the nodes of ${\mathcal Y}$,
which is a contradiction since $D\in{\mathcal Y}.$

Now, let us verify the ``Moreover" statement. Suppose ten nodes in $\mathcal S\subset{\mathcal X}$ use the line ${\ell}.$ Then, as we obtained earlier, the nodes ${\mathcal Y}:={\mathcal X}\setminus (\mathcal A\cup\mathcal S)$ are located in a
maximal curve $\mu$ of degree $n-4.$ Therefore the fundamental polynomial of each $A\in \mathcal S$ uses $\mu$ and hence $\ell:$
$$\ p^\star_A=\mu \ell q_{A},\ \hbox{where}\ q_{A}\in\Pi_3.$$
It is easily seen that $q_{A}$ is a $3$-fundamental polynomial of $A\in \mathcal S.$
\end{proof}

\vskip 12pt

%%
%% Here goes the Bibliography
%%

{\def\section*#1{}
\vskip 5pt
\begin{center}
{\bf REFERENCES}
\end{center}
%\vskip 2pt

%\end{document}

\end{document}